\documentclass[reqno]{amsart}
\usepackage[T1]{fontenc}
\usepackage[c5paper, text={126truemm,178truemm},centering]{geometry}

\def\thead{}
\def\doiCode{10.1007/s00022-018-0426-2}\def\thead{doi: {\href{http://dx.doi.org/\doiCode}{\doiCode}}}

\makeatletter
\def\@seccntformat#1{%
  \protect\textup{%
    \protect\@secnumfont
    \expandafter\protect\csname format#1\endcsname 
    \csname the#1\endcsname
    \protect\@secnumpunct
  }%
}
\makeatother

\usepackage{xpatch}
\xpatchcmd{\section}{\scshape}{\bfseries\scshape}{}{}
\xpatchcmd{\subsection}{\bfseries}{\bfseries\scshape}{}{}
\xpatchcmd{\subsubsection}{\itshape}{\scshape}{}{}
\xpatchcmd{\proof}{\itshape}{\bfseries}{}{}

\usepackage{datetime,float,amsmath,amsthm,amssymb}

\allowdisplaybreaks

\numberwithin{equation}{section}
\numberwithin{figure}{section}

\usepackage{xcolor}

\usepackage{tikz}
\usetikzlibrary{calc,arrows,decorations.markings,intersections,matrix}
\newcommand\includetikz[2][x=1mm,y=1mm]{%
 \IfFileExists{#2.tikz}{%
	\begin{tikzpicture}[#1]\node at (0,0) {\input{#2.tikz}};\end{tikzpicture}%
 }{%
	\@latex@error{No usable file `#2.tikz' can be found}%
	 {I could not locate the file...}%
 }
}

\usepackage[all]{background}
 \SetBgColor{black}%
 \SetBgOpacity{1.0}%
 \SetBgScale{1}
 \SetBgPosition{current page.south}%
 \SetBgAngle{0.0}%
 \SetBgVshift{10pt}%
 \SetBgContents{\scalebox{0.7}{\color{blue!70}%
   Journal of Geometry, 109(2018), 22 
	 \ \copyright\ \shortauthors\ 
\href{http://www.math.u-szeged.hu/tagok/kurusa}{http://www.math.u-szeged.hu/tagok/kurusa}%
 }}

\usepackage[font=small]{caption} 

\usepackage[inline]{enumitem}
\setlist[enumerate,1]{label={\upshape(\arabic*)}}

\providecommand\textqq[1]{``#1''}

\usepackage{hyperref}
\hypersetup{colorlinks, citecolor=[rgb]{0.0,0.3,0.1}, linkcolor=[rgb]{0.3,0.0,0.1},
						urlcolor=[rgb]{0.0,0.0,0.4}}

\theoremstyle{plain}
\newtheorem{theorem}{Theorem}[section]
\newtheorem{lemma}[theorem]{Lemma}

\newtheorem{corollary}[theorem]{Corollary}

\theoremstyle{definition}
\newtheorem*{acknowledgement}{Acknowledgement}

\def\Bf#1{\ifmmode\boldsymbol{#1}\else{\rmfamily\bfseries#1}\fi}
\DeclareMathOperator{\sign}{sign}

\makeatletter 
\AtBeginDocument{%
 \hypersetup{%
  pdftitle={\@title},
  pdfauthor={\authors},
  pdfsubject={\@subjclass},
  pdfkeywords={\@keywords}
 }
 \begin{picture}(0,0)(0,0)\put(340,50){\hbox to 0mm{\hss\small\itshape\thead}}\end{picture}
 \null\vskip-12mm\null%
}%
\makeatother 

\title[Straight projective-metric spaces with centers]
{Straight projective-metric spaces with centers}



\author[\'A. Kurusa]{\'Arp\'ad Kurusa}
\thanks{This research was supported by 
	NFSR of Hungary (OTKA), grant number K~116451}
\address{\'A. Kurusa,
\upshape Bolyai Institute, University of Szeged,
Aradi v\'ertan\'uk tere 1, 6725 Szeged (Hungary);
E--mail: {\tt kurusa@math.u-szeged.hu }.
}

\begin{document}%

\begin{abstract}
 It is proved that a straight projective-metric space has an open set of centers
 if and only if it is either the hyperbolic or a Minkowskian geometry.
 It is also shown that if a straight projective-metric space has some finitely many
 well-placed centers,
 then it is either the hyperbolic or a Minkowskian geometry.
\end{abstract}


\maketitle

\section{Introduction}\label{sec:intro}

Let $(\mathcal M,d)$ be a metric space given in a set $\mathcal M$
with the metric~$d$.
If $\mathcal M$ is a projective space $\mathbb P^n$
 or an affine space $\mathbb R^n\subset\mathbb P^n$
 or a proper open convex subset of~$\mathbb R^n$ for some $n\in\mathbb N$,
and
 the metric $d$ is complete,
 continuous with respect to the standard topology of $\mathbb P^n$,
and
 the geodesic lines of $d$ are exactly
 the non-empty intersection of $\mathcal M$ with the straight lines,
then the metric~$d$ is called \emph{projective}.

If 
 $\mathcal M=\mathbb P^n$ and
 the geodesic lines of $d$ are isometric with a Euclidean circle,
or
 $\mathcal M\subseteq\mathbb R^n$ and
 the geodesic lines of $d$ are isometric with a Euclidean straight line,
then
 $(\mathcal M,d)$ is called a \emph{projective-metric space}
 of dimension~$n$
(see \cite[p.~115]{BusemannKelly1953} and~\cite[p.~188]{Szabo})

Such projective-metric spaces 
are called of
\emph{elliptic, parabolic or hyperbolic type}
according to whether $\mathcal M$ is
$\mathbb P^n$, $\mathbb R^n$, or a proper convex subset of~$\mathbb R^n$.
The projective-metric spaces of the latter two types are called \emph{straight}
\cite[p.~1]{Busemann1955}.

A \emph{metric point reflection $\rho_{d;O}$} of a projective-metric space $(\mathcal M,d)$
is a non-identical, involutive $d$-isometry of $\mathcal M$ onto $\mathcal M$
such that it fixes point $O$ and keeps every geodesic line passing through~$O$.
A \emph{center} of the projective-metric space $(\mathcal M,d)$ is 
a point $O\in\mathcal M$, where there exists a metric point reflection $\rho^{}_{d;O}$.
If every point of a projective-metric space is a center,
then it is said to be \emph{symmetric}.

\emph{What are the symmetric projective-metric spaces?}


Working with $G$-spaces\footnote{Projective-metric spaces
are the Desarguesian $G$-spaces \cite[p.~188]{Szabo}.},
Busemann proved in \cite{Busemann1955} that
a symmetric $G$-space of elliptic type is elliptic \cite[(49.5)]{Busemann1955}\footnote{%
\textqq{There is no similar theorem for straight $G$-spaces}
as Busemann proves on \cite[p.~346]{Busemann1955}.},
and a symmetric $G$-space of dimension~$2$ is either elliptic,
or hyperbolic, or Minkowskian \cite[(52.8)]{Busemann1955}.

In this paper we complement Busemann's results
for straight projective-metric spaces
by proving \emph{directly} in \emph{every dimension}
 (Theorem~\ref{thm:PPMcent} and Theorem~\ref{thm:HPMcent})
that
a straight projective-metric space
has a non-empty open set of centers
if and only if it is
a Minkowskian or the hyperbolic geometry, respectively.
Further, we show
 (Theorems~\ref{thm:strPMcent}, \ref{thm:parPMcentMinkowski} and \ref{thm:hypPMcentBolyai})
that
a straight projective-metric space
has some finitely many well-placed centers
if and only if it is
a Minkowskian or the hyperbolic geometry, respectively.


\section{Notations and preliminaries}

Points of $\mathbb R^n$ are denoted as $A,B,\dots$,
vectors are $\overrightarrow{AB}$ or $\Bf a , \Bf b , \dots$.
Latter notations are also used for points if the origin is fixed.
Open segment with endpoints $A$ and $B$ is denoted by $\overline{AB}$ 
and the line through $A$ and $B$ is denoted by $AB$.
The Euclidean scalar product is $\langle \cdot ,\cdot \rangle$.

The \emph{affine ratio} $(A,B;C)$ of the collinear points $A$, $B$ and $C\ne B$
is defined by $(A,B;C)\overrightarrow{BC}=\overrightarrow{AC}$.
The \emph{affine cross ratio} of the collinear points $A$, $B$, $C\ne B$,
and $D\ne A$ is $(A,B;C,D)=(A,B;C)/(A,B;D)$ \cite[page 243]{BusemannKelly1953}.
The \emph{affine point reflection} $\bar\rho^{}_{O}$ at point $O$
is defined by 
$(X,\bar\rho^{}_{O}(X);O)=-1$ for every point $X\ne O$
and by $\bar\rho^{}_{O}(O)=O$.

According to \cite[p.~64]{MontejanoMorales2003},
a point $O$ is called a \emph{projective center} of the set $\mathcal S\subseteq\mathbb P^n$,
if there is a projectivity $\varpi$ 
such that $\varpi(O)$ is the affine center of $\varpi(\mathcal S)$.

Fix a point $O$ in the convex open bounded domain $\mathcal D\subseteq\mathbb R^n$.
We define $O^*$ as the locus of every point $P$ which is
the harmonic conjugate\footnote{Point $P$ satisfies $(A,B;O,P)=-1$.}
of $O$ with respect to points $A$ and $B$,
where $\{A,B\}=\partial\mathcal D\cap OP$.
It is easy to see that
a point $O$ is a projective center of $\mathcal D$
if and only if
$O^*$ is a straight line that does not intersect $\mathcal D$
\cite[p.~64]{MontejanoMorales2003}.

%



If a projective-metric space  $(\mathcal M,d)$ is given,
we denote the geodesic line on the projective line $\ell$ by $\tilde\ell$,
i.e. $\tilde\ell=\ell\cap\mathcal M$.

\subsection{Projectively invariant metrics on projective lines}
The following easy, perhaps folkloric results are provided here for the sake of completeness,
and because the author could not find a really good reference for them.

\begin{lemma}\label{lem:PMonSegment}
 Let the function $h\colon(a,b)\times(a,b)\to\mathbb R$ be
 such that $h(x,y)+h(y,z)=h(x,z)$ and $h(x,y)=h(\varpi(x),\varpi(y))$
 for  every $x,y,z\in(a,b)$ and
 any projectivity $\varpi\colon(a,b)\to(a,b)$.
 If $h$ is bounded on an open subset of $(a,b)\times(a,b)$,
 then there exists a constant  $c\in\mathbb R$ such that
 \[h(x,y)=c|\ln(a,b;x,y)|.\]
\end{lemma}
\begin{proof}
Let $\mathbb R_{+}=\{x\in\mathbb R:x>0\}$ 
and fix the projectivity 
$\omega\colon x\in(a,b)\mapsto\frac{x-a}{b-x}\in\mathbb R_{+}$.
Then the function
$f\colon(x,y)\!\in\!\mathbb R_{+}\!\times\!\mathbb R_{+}
 \mapsto h(\omega^{-1}(x),\omega^{-1}(y))\!\in\!\mathbb R$
clearly satisfies $f(x,y)+f(y,z)=f(x,z)$ and $f(x,y)=f(\hat\varpi(x),\hat\varpi(y))$
for  every $x,y,z\in\mathbb R_{+}$ and
for any surjective projectivity $\hat\varpi\colon\mathbb R_{+}\to\mathbb R_{+}$,
as $\omega^{-1}\circ\hat\varpi\circ\omega\colon(a,b)\to(a,b)$
is a surjective projectivity.

A projectivity of an affine straight line 
is an affinity.
An affinity of a straight line with a fixed point is a dilation,
thus we have
$f(x,y)=f(bx,by)$ for every $x,y\in\mathbb R_+$ and $b\in\mathbb R_+$.
Choosing $b=1/x$ implies $f(x,y)=f(1,y/x)$ for every $x,y\in\mathbb R_+$,
from which $f(1,y/x)+f(1,z/y)=f(x,y)+f(y,z)=f(x,z)=f(1,z/x)$ follows
for every $x,y\in\mathbb R_+$,
hence $f(1,s)+f(1,t)=f(1,st)$ for every $s,t\in\mathbb R_+$.

Let $g(u)=f(1,e^u)$ for every $u\in\mathbb R_+$.
Then $g(p)+g(q)=g(p+q)$ for every $p,q\in\mathbb R_+$,
hence, by known properties of Cauchy's functional equation \cite{WP-CauchyEqu},
$g(u)=cu$ follows for some $c\in\mathbb R$ and every $u\in\mathbb R_+$.
Thus
\begin{equation*}
 \begin{split}
  h(x,y)
   &=f(\omega(x),\omega(y))
    =f(1,\omega(y)/\omega(x))
    =g(\ln(\omega(y)/\omega(x)))
  \\
   &=c\cdot\ln(\omega(y):\omega(x))
    =c\ln\Big(\frac{y-a}{b-y}:\frac{x-a}{b-x}\Big).\qedhere
 \end{split}
\end{equation*}
\end{proof}

\def\looseBUTDONTforget{%
\begin{lemma}\label{lem:PMongcircle}
 Fix an $m>0$ and let $e\colon\mathbb P\times\mathbb P\to[0,m]$
 be a continuous function that satisfies 
 $e([\Bf x],[\Bf y])=0$ if and only if $[\Bf x]=[\Bf y]$,
 fulfills
 \[e([\Bf x],[\Bf y])+e([\Bf y],[\Bf z])
   =\begin{cases}
     e([\Bf x],[\Bf z]),   &\text{ if $e([\Bf x],[\Bf y])+e([\Bf y],[\Bf z])\le m$},\\
     2m-e([\Bf x],[\Bf z]),&\text{ if $e([\Bf x],[\Bf y])+e([\Bf y],[\Bf z])> m$},
    \end{cases}
 \]
 and satisfies $e([\Bf x],[\Bf y])=e(\varpi([\Bf x]),\varpi([\Bf y]))$
 for every surjective projectivity $\varpi\colon\mathbb P\to\mathbb P$
 that has no fixed point.
 Then
 \[e([\Bf u_\xi],[\Bf u_\psi])=2m\cdot
   \begin{cases}
     |\tilde\xi-\tilde\psi|/\pi,  &\text{ if $|\tilde\xi-\tilde\psi|\in[0,\pi]$},\\
     1-|\tilde\xi-\tilde\psi|/\pi,&\text{ if $|\tilde\xi-\tilde\psi|\in[\pi,2\pi]$},
   \end{cases}
 \]
 where $\tilde\zeta=\zeta-2\pi\big[\frac{\zeta}{2\pi}\big]$.
\end{lemma}
\begin{proof}
%
%
A surjective projectivity $\varpi\colon\mathbb P^{1}\to\mathbb P^{1}$
is of the form $[\Bf x]\mapsto[\Bf x\Bf M]$,
where $\Bf M$ is a non-degenerate matrix of type $2\times2$.
If $\varpi$ has no fixed point, then $\Bf M$ does not have eigenvector,
hence it is a rotation, that is,
$\varpi\colon\Bf u_\xi\mapsto\big[\Bf u_\xi
\big(\begin{smallmatrix}\cos\varphi&-\sin\varphi\\\sin\varphi&\cos\varphi\end{smallmatrix}\big)
\big]=[\Bf u_{\xi+\varphi}]$
for some $\varphi\in\mathbb R$.
Thus, we have
$
e([\Bf u_\xi],[\Bf u_\psi])=e([\Bf u_{\xi+\varphi}],[\Bf u_{\psi+\varphi}])
$
for every $\xi,\psi,\varphi\in\mathbb R$, that implies
$ 
e([\Bf u_\xi],[\Bf u_\psi])=e([\Bf u_0],[\Bf u_{\psi-\xi}])
$ 
for every $\xi,\psi\in\mathbb R$.
Let $f(\sigma)=e([\Bf u_0],[\Bf u_{\sigma}]))/m$ for every $\sigma\in\mathbb R$.

Then $f$ is periodic by $2\pi$, continuous, non-negative, and
\begin{equation}\label{eq:kindofadditivity}
f(\sigma)+f(\theta)
  =\begin{cases}f(\sigma+\theta),&\text{ if $f(\sigma)+f(\theta)\le1$},\\
                2-f(\sigma+\theta),&\text{ if $f(\sigma)+f(\theta)>1$},\end{cases}
\end{equation}
for every $\sigma,\theta\in\mathbb R$.
Further, $f(\xi)$ vanishes only for $\xi=2k\pi$, where $k\in\mathbb Z$.

As $f$ is continuous and $f(0)=0$,
there is a $\delta>0$ such that $f(\xi)<1/2$ and $f(\xi)\ne0$ if $\xi\in(0,\delta)$.
Fix a $\xi\in(0,\delta)$.

Then, by \eqref{eq:kindofadditivity},
we have $1\ge2f(\xi/2)=f(\xi)$,
hence $2^if(\xi/2^i)=f(\xi)$ follows for every $i\in\mathbb N$.
By this, for any $c=\sum_{i=1}^\infty c_i2^{-i}$, where $c_i\in\{0,1\}$,
we have
\[
  1\ge cf(\xi)=\sum_{i=1}^\infty c_i2^{-i}f(\xi)
  =\sum_{i=1}^\infty f(c_i2^{-i}\xi)
  =g\Big(\sum_{i=1}^kc_i2^{-i}\xi\Big)+\sum_{i=k+1}^\infty f(c_i2^{-i}\xi)
\]
for every $k\in\mathbb N$.
As the last sum goes to zero with $k\to\infty$,
$cf(\xi)=f(c\xi)$ follows.
For $c\le1/f(\xi)$ we obtain
\begin{align*}
 1\ge cf(\xi)&=([c]+\{c\})f(\xi)=[c]f(\xi)+\{c\}f(\xi)=f([c]\xi)+f(\{c\}\xi)
 \\&=f(([c]+\{c\})\xi)=f(c\xi),
\end{align*}
where  $[c]$ is the integer part, and $\{c\}$ is the fractional part of $c$.
Continuing by considering $c\in[1/f(\xi),2/f(\xi)]$, we obtain
\[
 1\le cf(\xi)=2\frac{c}{2}f(\xi)=2g\Big(\frac{c}{2}\xi\Big)
 =2-f(c\xi).
\]
By the above, letting $p=\xi/f(\xi)$ gives $f(p)=1$ and $f(2p)=0$,
hence $p=2\pi$.
This proves the lemma.
\end{proof}
}

\begin{lemma}\label{lem:PMonsline}
 Let $m\colon\mathbb R\times\mathbb R\to\mathbb R$
 be a continuous function that
 satisfies $m(x,y)=0$ if and only if $x=y$,
 fulfills $m(x,y)+m(y,z)=m(x,z)$ if and only if $y\in[x,z]$,
 and satisfies $m(x,y)=m(\varpi(x),\varpi(y))$
 for every surjective projectivity $\varpi\colon\mathbb R\to\mathbb R$
 that has no fixed point.
 Then there are constants $c_+,c_-\in\mathbb R$ such that
 \[m(x,y)=c_{\sign(x-y)}|y-x|.\]
\end{lemma}
\begin{proof}
A projectivity of an affine straight line is an affinity.
An affinity of $\mathbb R$ has the form $\varpi\colon x\mapsto ax+b$,
where $a,b\in\mathbb R$.
If $a\ne1$, then $\frac{b}{1-a}$ is a fixpoint of $\varpi$,
therefore, we have $\varpi\colon x\mapsto x+b$.
Thus we have
$m(x,y)=m(x+b,y+b)$ for every $x,y\in\mathbb R$ and $b\in\mathbb R$.

Choosing $b=-x$, we obtain $m(x,y)=m(0,y-x)$ for every $x,y\in\mathbb R$,
from which $m(0,y-x)+m(0,z-y)=m(x,y)+m(y,z)=m(x,z)=m(0,z-x)$,
hence $m(0,s)+m(0,t)=m(0,s+t)$ follows for every $s,t\in\mathbb R_+$.
Let $g(u)=m(0,u)$ for every $u\in\mathbb R_+$.
Then $g(p)+g(q)=g(p+q)$ for every $p,q\in\mathbb R_+$,
hence, by known properties of Cauchy's functional equation \cite{WP-CauchyEqu},
$g(u)=c_+u$ follows for some $c_+\in\mathbb R$ and every $u\in\mathbb R_+$.
Thus, $m(x,y)=m(0,y-x)=g(y-x)=c_+(y-x)$ for every $x\le y$.

Now choose $b=-y$ to obtain $m(x,y)=m(x-y,0)$ for every $x,y\in\mathbb R$,
implying $m(x-y,0)+m(y-z,0)=m(x,y)+m(y,z)=m(x,z)=m(x-z,0)$,
hence $m(s,0)+m(t,0)=m(s+t,0)$ for every $s,t\in\mathbb R_-$.
Let $g(u)=m(u,0)$ for every $u\in\mathbb R_-$.
Then $g(p)+g(q)=g(p+q)$ for every $p,q\in\mathbb R_-$,
hence, by known properties of Cauchy's functional equation \cite{WP-CauchyEqu},
$g(u)=c_-u$ follows for some $c_-\in\mathbb R$ and every $u\in\mathbb R_-$.
Thus, $m(x,y)=m(y-x,0)=g(x-y)=c_-(x-y)$ for every $y\le x$.
\end{proof}


\subsection{Straight projective-metric spaces}
The following two (most) important examples
are distinguished among the straight projective-metric spaces
by the property that an isometry of one geodesic on another
or itself is a projectivity.

\subsubsection{Minkowski geometry}

Given an open, strictly convex, bounded domain  $\mathcal I\subset\mathbb R^n$,
the \emph{indicatrix},
that is (centrally) symmetric to the origin,
the function $d_{\mathcal I}\colon\mathbb R^n\times\mathbb R^n\to\mathbb R$
defined by
\begin{equation*} 
 d_{\mathcal I}(\Bf x ,\Bf y )
 =\inf\{\lambda>0:(\Bf y-\Bf x)/\lambda\in\mathcal I\}
\end{equation*}
is a metric on $\mathbb R^n$ \cite[VI.48]{BusemannKelly1953},
and is called \emph{Minkowski metric} on $\mathbb R^n$.
The projective-metric spaces of type $(\mathbb R^n,d_{\mathcal I})$
are all called \emph{Minkowski geometry}.
It is the Euclidean geometry
if and only if $\mathcal I$ is an ellipsoid \cite[(48.7)]{BusemannKelly1953}.



\subsubsection{Hilbert geometry}

Given an open, strictly convex, bounded domain  $\mathcal I\subset\mathbb R^n$,
that does not contain two coplanar non-collinear segments,
the function $d_{\mathcal I}\colon\mathcal I\times\mathcal I\to\mathbb R$
defined by
\begin{equation*}
 d_{\mathcal I}(A,B)
 =\begin{cases}
   0,&\text{if $A=B$},\\
   \frac{1}{2}\bigl|\ln(A,B;C,D)\bigr|,
     &\text{if $A\ne B$, where $\overline{CD}=\mathcal I\cap AB$},
  \end{cases}
\end{equation*}
is a projective metric on $\mathcal I$ \cite[VI.50]{BusemannKelly1953},
and is called the \emph{Hilbert metric} on $\mathcal I$.
The projective-metric space $(\mathcal I,d_{\mathcal I})$ is called \emph{Hilbert geometry}
given in $\mathcal I$.
It is the hyperbolic geometry
if and only if $\mathcal I$ is an ellipsoid \cite[(50.2)]{BusemannKelly1953}.




\subsection{Isometries and metric point reflections}
Although some of the statements here are valid more generally,
we confine ourselves here to straight projective-metric spaces $(\mathcal M,d)$.

An isometry  
keeps the geodesic lines,
therefore, it is the restriction of a collineation \cite[Theorem 3.1.]{CCMECallum2007}.
A collineation is,
by Staudt's theorem \cite[(ii) Fundamental theorem of projective geometry, p.~30]{Hirschfeld1979},
a projective map of $\mathbb P^n$, so we obtain that
\begin{equation}\label{obs:isomProj} 
 \parbox{100mm}{\noindent\itshape%
 isometries 
 are restrictions of projective maps. 
 }
\end{equation}
If there exists a metric point reflection $\rho^{}_{d;O}$,
then 
$O$ is the metric midpoint of the geodesic segment
$\overline{P\rho^{}_{d;O}P}\subset\widetilde{P\rho^{}_{d;O}P}$
for any point $P\ne O$, hence we have that
\begin{equation}\label{obs:preflUniq} 
 \parbox{100mm}{\noindent\itshape%
 there is at most one metric point reflection at every point.
 }
\end{equation}
The easy formal proof of the following statement is left to the reader.
\begin{equation}\label{obs:CentMapsCent} 
 \parbox{100mm}{\noindent\itshape%
 If $O$ is a center,
 and $\varpi$ is a projective map of $\mathbb P^n$,
 then $\varpi(O)$ is a center of the projective-metric space
 $(\varpi(\mathcal M),d_\varpi)$,
 where $d_\varpi(\varpi(P),\varpi(Q))=d(P,Q)$
 for every point $P,Q\in\mathcal M$.
 }
\end{equation}


\begin{lemma}\label{lem:genCent} 
 We have
   $\rho^{}_{d;\rho^{}_{d;O}Q}=\rho^{}_{d;O}\circ\rho^{}_{d;Q}\circ\rho^{}_{d;O}$.
\end{lemma}
\begin{proof}
The map $\imath^{}_{}:=\rho^{}_{d;O}\circ\rho^{}_{d;Q}\circ\rho^{}_{d;O}$
is clearly a non-trivial isometry,
and fixes point $Q':=\rho^{}_{d;O}Q$, because
$\imath^{}_{}Q'=\imath^{}_{}\rho^{}_{d;O}Q=\rho^{}_{d;O}\rho^{}_{d;Q}Q=\rho^{}_{d;O}Q=Q'$.
Further, it satisfies
$\imath^2=(\rho^{}_{d;O}\circ\rho^{}_{d;Q}\circ\rho^{}_{d;O})\circ
(\rho^{}_{d;O}\circ\rho^{}_{d;Q}\circ\rho^{}_{d;O})=\mathop{\rm id}
$.

Assume that points $A',B'\in\mathcal M$ are such that $Q'\in\widetilde{A'B'}$.
Let $A=\rho^{}_{d;O}A'$ and $B=\rho^{}_{d;O}B'$.
Then
$Q=\rho^{}_{d;O}Q'\in\rho^{}_{d;O}\big(\widetilde{A'B'})
 =\widetilde{AB}$,
hence
$\imath^{}_{}\big(\widetilde{A'B'}\big)=\widetilde{\imath^{}_{}A'\imath^{}_{}B'}
 =\rho^{}_{d;O}\circ\rho^{}_{d;Q}\big(\widetilde{AB}\big)
 =\rho^{}_{d;O}\big(\widetilde{BA}\big)=\widetilde{B'A'}$,
i.e.\ $\imath^{}_{}$ keeps every geodesic passing through~$Q'$.

Thus, $\imath^{}_{}$ is non-trivial, isometric, fixes $Q'$, involutive,
and keeps the geodesic lines passing through $Q'$,
therefore, by \eqref{obs:preflUniq}, it is the metric point reflection $\rho_{d;Q'}$.
\end{proof}

\begin{lemma}\label{lem:PRSclosed}
 The set of the centers 
 is closed.
\end{lemma}
\begin{proof}
 Let $O_n$ be a sequence of centers of the projective-metric space $(\mathcal M,d)$
 converging to $O_\infty$.
 Then we have the sequence of points $P_n=\rho^{}_{d;O_n}(P)$ for any point $P\in\mathcal M$.
 
 From $d(P_n,O_n)=d(P,O_n)$, $O_n\to O_\infty$ and the triangle inequality
 it follows that
 $d(P_n,O_\infty)\le d(P_n,O_n)+d(O_n,O_\infty)=d(P,O_n)+d(O_n,O_\infty)
  \le d(P,O_\infty)+d(O_\infty,O_n)+d(O_n,O_\infty)\le d(P,O_\infty)+\varepsilon$
 for any $\varepsilon>0$ if $n\in\mathbb N$ is big enough.
 Thus, the sequence of points $P_n$ is bounded,
 hence it has congestion points.
 
 If $P_\infty$ is a congestion point of $P_n$,
 then
 \[
 d(P,O_\infty)+d(O_\infty,P_\infty)=\lim_{n\to\infty}(d(P,O_n)+d(O_n,P_n))
 =\lim_{n\to\infty}d(P,P_n)=d(P,P_\infty)
 \]
 proves that $O_\infty\in\widetilde{PP_\infty}$,
 and 
 $ 
 d(P,O_\infty)=\lim_{n\to\infty}(d(P,O_n)=\lim_{n\to\infty}(d(P_n,O_n)=d(P_\infty,O_\infty)
 $ 
 proves that $O_\infty$ is the metric midpoint of the segment $\overline{PP_\infty}$.
 Thus, $P_\infty=\rho^{}_{d;O_\infty}(P)$, hence the Lemma.
\end{proof}

Two point reflections define an isometry defined by
$\tau^{}_{PQ}:=\rho^{}_{d;P}\circ\rho^{}_{d;Q}$.
We call such isometries \emph{translations}.

\begin{lemma}\label{lem:segments}
  For any three collinear points $O,P,Q$
  \begin{enumerate}
   \item\label{enu:inIFFin}
    $O\in\overline{PQ}$ if and only if $O\in\overline{\rho^{}_{d;P}(O)\rho^{}_{d;Q}(O)}$,
    and
   \item\label{enu:transLength}
    $d(\tau^{}_{PQ}(O),O)=2d(P,Q)$.
  \end{enumerate}
\end{lemma}
\begin{proof}
To prove \ref{enu:inIFFin} we need only to observe that
the points $P$ and $Q$ are on the same side of $O$ as
the points $\rho^{}_{d;P}(O)$ and $\rho^{}_{d;Q}(O)$,
respectively.

For \ref{enu:transLength} we let
$\delta:=d(\tau^{}_{PQ}(O),O)=d(\rho^{}_{d;P}(\rho^{}_{d;Q}(O)),O)
=d(\rho^{}_{d;Q}(O),\rho^{}_{d;P}(O))$,
and consider three cases:
\begin{enumerate}[label={\upshape(\alph*)}]
 \item
  if $O\in\overline{PQ}$, then $O\in\overline{\rho^{}_{d;P}(O)\rho^{}_{d;Q}(O)}$,
  hence\\
  $\delta
   =d(\rho^{}_{d;Q}(O),O)+d(O,\rho^{}_{d;P}(O))
   =2d(Q,O)+d(O,P)=2d(P,Q)$;
 \item
  if $P\in\overline{OQ}$, then $\rho^{}_{d;P}(O)\in\overline{O\rho^{}_{d;Q}(O)}$,
  hence\\
  $\delta
     =d(\rho^{}_{d;Q}(O),O)-d(O,\rho^{}_{d;P}(O))
     =2d(Q,O)-d(O,P)=2d(P,Q)$;
 \item
  if $Q\in\overline{OP}$, then $\rho^{}_{d;Q}(O)\in\overline{O\rho^{}_{d;P}(O)}$,
    hence\\
  $\delta
     =d(\rho^{}_{d;P}(O),O)-d(O,\rho^{}_{d;Q}(O))
     =2d(P,O)-d(O,Q)=2d(P,Q)$.\qedhere
\end{enumerate}

\end{proof}

\begin{lemma}\label{lem:CCTrans}
 Assume that every point of the geodesic line $\tilde\ell$ is a center.
 Then every isometry of $\tilde\ell$ is a restriction 
 of $\tau^{}_{PQ}$ or $\rho_{d;P}$,
 where $P,Q$ are any points on $\tilde\ell$.
\end{lemma}
\begin{proof}
By definition we have an isometry $\imath\colon\tilde\ell\to\mathbb R$.

If $\jmath$ is an isometry on $\tilde\ell$,
then $\imath\circ\jmath\circ\imath^{-1}$ is an isometry on $\mathbb R$.
Every isometry on $\mathbb R$ has the form of
either $x\mapsto a+x$ or $x\mapsto a-x$ for some $a\in\mathbb R$,
so we have for a fixed $a\in\mathbb R$
either $\imath(\jmath(\imath^{-1}(x)))=a+x$ or $\imath(\jmath(\imath^{-1}(x)))=a-x$
for every $x\in\mathbb R$.

Thus, every isometry $\jmath$ on $\tilde\ell$ is 
either $\jmath(P)=\imath^{-1}(a+\imath(P))$ or $\jmath(P)=\imath^{-1}(a-\imath(P))$,
for some $a\in\mathbb R$.

If $\jmath(\cdot)=\imath^{-1}(a+\imath(\cdot))$,
then
$
 d(\jmath(P),P)
 =|\imath(\jmath(P))-\imath(P)|
 =|a+\imath(P)-\imath(P)|=a
$,
hence, by Lemma~\ref{lem:segments}\ref{enu:transLength},
$\jmath=\tau^{}_{QR}$, where $Q,R\in\tilde\ell$ and $d(Q,R)=a/2$.

If $\jmath(\cdot)=\imath^{-1}(a-\imath(\cdot))$,
then
we have a point $O\in\tilde\ell$ such that $\imath(O)=a/2$,
and 
$
 d(\jmath(P),O)
 =|\imath(\jmath(P))-\imath(O)|
 =|a/2-\imath(P)|
 =|\imath(P)-\imath(O)|
 =d(P,O)
$,
as well as
$
 d(\jmath(P),P)
 =|\imath(\jmath(P))-\imath(P)|
 =2|a/2-\imath(P)|
 =2d(P,O)
$,
hence, by \eqref{obs:preflUniq}, $\jmath=\rho^{}_{O}$.
\end{proof}

\section{Open set of centers}

Firstly, we note the well-known fact that
\begin{equation}\label{equ:symmSPMS}
 \parbox{100mm}{\noindent\itshape%
 Minkowski geometries and the hyperbolic geometry are symmetric.
 }
\end{equation}


\begin{theorem}\label{thm:PPMcent}
 The set of the centers of a projective-metric space of parabolic type
 contains a non-empty open set of centers
 if and only if it is Minkowskian geometry.
\end{theorem}
\begin{proof}
By \eqref{equ:symmSPMS} and Lemma~\ref{lem:genCent},
we only need to prove that
if every point of a projective-metric space of parabolic type is a center,
then it is a Minkowskian geometry.

First, we prove that
\begin{equation}\label{equ:PMCisAC}
 \parbox{100mm}{\noindent\itshape%
  if $O$ is a center of a projective-metric space of parabolic type,
  then the metric point reflection $\rho^{}_O$ is the
  affine point reflection $\bar\rho^{}_O$.
 }
\end{equation}

Let the straight line $\ell$ avoid $O$ and let $\ell'=\rho^{}_O(\ell)$.
As $\rho^{}_O$ keeps the straight lines containing $O$,
every straight line $l$ through $O$ and a point $P$ of $\ell$
coincides $\rho^{}_Ol$.
As all these lines are in the common plane $\mathbb R^2_{O,\ell}$ of $O$ and $\ell$,
we conclude that $\ell$ and $\ell'$ are in $\mathbb R^2_{O,\ell}$.

Assume that $\ell$ intersects $\ell'$,
i.e. there is a point $P$ in $\ell\cap\ell'$. 
Then $\rho^{}_O(P)$ is also in $\ell\cap\ell'$ and is different from $P$
as $O$ is the metric midpoint of the segment $\overline{P\rho^{}_O(P)}$,
and $d(O,P)>0$.
Thus, $\ell$ and $\ell'$ have two different common points,
hence $\ell\equiv\ell'$.
This is a contradiction as $O\in\overline{P\rho^{}_O(P)}\subset\ell$,
but $O\notin\ell$.
Thus, $\ell$ does not intersect $\ell'$, that,
as these straight lines are in their common plane $\mathbb R^2_{O,\ell}$,
implies that $\ell\parallel\ell'$.
So, $\rho^{}_O$ maps every straight line into a parallel straight line.

Let $O$ and $A$ be arbitrary different points.
Let $B$ be any point outside their common straight line. 
By the above observation $AB\parallel\rho^{}_O(A)\rho^{}_O(B)$
and $A\rho^{}_O(B)\parallel\rho^{}_O(A)B$,
hence quadrangle $\mathcal P:=AB\rho^{}_O(A)\rho^{}_O(B)\Box$ is a parallelogram.
As $O$ is the intersection of the diagonals of $\mathcal P\,$,
it follows that $O$ is the affine midpoint of the segments
$\overline{A\rho^{}_O(A)}$. This proves~\eqref{equ:PMCisAC}.

Let $A$ and $B$ be arbitrary different points,
and let $O$ be the $d$-metric midpoint of segment $\overline{AB}$.
Then $\rho^{}_O(A)=B$, and by \eqref{equ:PMCisAC},
$O$ is the affine midpoint of $\overline{AB}$ too.

Thus,
the affine midpoint and the $d$-metric midpoint of any segment coincide
which, by \cite[(17.9)]{Busemann1955}, implies that $d$ is a Minkowskian metric.
\end{proof}


\begin{theorem}\label{thm:HPMcent}
 The set of the centers of a projective-metric space of hyperbolic type
 contains a non-empty open set
 if and only if it is the hyperbolic geometry.
\end{theorem}
\begin{proof}
By \eqref{equ:symmSPMS} and Lemma~\ref{lem:genCent},
we only need to prove that
if every point of a projective-metric space $(\mathcal M,d)$ of hyperbolic type
is a center,
then it is the hyperbolic geometry.

By \cite[Lemma 12.1, pp.~226]{BusemannKelly1953},
a bounded open convex set $\mathcal I$ in $\mathbb R^n$ ($n\ge2$)
is an ellipsoid if and only if every section of it
by any 2-dimensional plane is an ellipse.
This means, that we only need to prove the statement
in dimension~$2$.\footnote{Although this is already proved in \cite[(52.8)]{Busemann1955},
we give here a more direct proof.}

As it is convex and proper subset of $\mathbb R^2$,
$\mathcal M$ cannot contain two intersecting affine straight line,
because otherwise it coincides with the affine plane $\mathbb R^2$.

\smallskip
\textit{Assume now that $\mathcal M$ contains an affine line.}

A convex domain in the plane which contains a straight line is
either a half plane or a strip bounded by two parallel lines
\cite[{{Exercise [17.8]}}]{BusemannKelly1953},
therefore, $\mathcal M$ is
either $\mathcal P_{(0,\infty)}:=\{(x,y)\in\mathbb R^2: 0<x\}$
or $\mathcal P_{(0,1)}:=\{(x,y)\in\mathbb R^2: 0<x<1\}$
in proper linear coordinatizations of $\mathbb R^2$.
As the perspective projectivity
$\varpi\colon(x,y)\mapsto\big(\frac{x}{x+1},\frac{y}{x+1}\big)$
maps $\mathcal P_{(0,\infty)}$ onto $\mathcal P_{(0,1)}$ bijectively,
\eqref{obs:CentMapsCent} immediately implies that
it is enough to consider the case $\mathcal M=\mathcal P_{(0,1)}$.

Suppose that every point of a projective-metric space $(\mathcal M,d)$
is a center.

By Lemma~\ref{lem:CCTrans}
the point reflections of $(\mathcal M,d)$ restricted onto a line $\tilde\ell$
generate every isometry of $\tilde\ell$,
and, by \eqref{obs:isomProj}, every point reflection of $(\mathcal M,d)$
is the restriction of a projective map of the projective plane onto $\mathcal M$,
hence Lemma~\ref{lem:PMonsline} gives that
$
 d((x,y),(x,z))=c(x)|z-y|
$
for a continuous functions $c\colon(0,1)\to\mathbb R_+$.
Function $c$ is a constant,
because the point reflection $\rho^{}_{d;(t,0)}$
maps $d$-isometrically the lines $\ell_x:=\{(x,y):y\in\mathbb R\}$ ($x\in(0,1)$)
onto lines $\ell_{z}$, where $\frac{1}z={1+\big(\frac{1-t}{t}\big)^2\frac{1-x}{x}}$.

In the same way as in the above paragraph, Lemma~\ref{lem:PMonSegment} gives
\[
 d((x,\lambda+\sigma x),(\mu x,\lambda+\mu\sigma x))
 =\bar c(\lambda,\sigma)\Big|\ln\Big(0,\frac1x;1,\mu\Big)\Big|
 =\bar c(\lambda,\sigma)\Big|\ln\frac{1-\mu x}{\mu(1-x)}\Big|,
\]
where $\bar c\colon\mathbb R\times\mathbb R_+\to\mathbb R_+$
is a continuous function.
Function $\bar c$ is a constant,
because the point reflection $\rho^{}_{d;O}$, where
$O=\big(\frac{\lambda}{2\lambda+\sigma},\frac{\lambda(\lambda+\sigma)}{2\lambda+\sigma}\big)$,
maps the open segment $\overline{(0,\lambda)(1,\lambda+\sigma)}$
onto $\overline{(0,0)(1,0)}$ $d$-isometrically.

By the aboves we have
\[
 d((x,0),(s,y))
 =\begin{cases}
 \bar c(0,0)\big|\ln\frac{x(1-s)}{s(1-x)}\big|,&\text{if $x\ne s$},\\
 c(1/2)|y|,&\text{if $x=s$},
 \end{cases}
\]
for every $x,s\in(0,1)$ and $y\in\mathbb R$,
hence
$
 c(1/2)|y|=\bar c(0,0)\lim_{x\to s}\big|\ln\frac{x(1-s)}{s(1-x)}\big|=0
$
by the continuity of $d$.
This contradiction proves that a projective-metric space
$(\mathcal P_{(0,1)},d)$
cannot be symmetric.
 
\smallskip
\textit{Assume now that $\mathcal M$ contains no affine line.}

Then every supporting line $\ell$ of $\mathcal M$ at any point  $M$ of
$\partial\mathcal M$ can intersect $\partial\mathcal M$ only in
a point, a segment or a ray.
Let $\ell^+$ be a straight line parallel to $\ell$
that is in the other side of $\ell$ than $\mathcal M$ is.
Now, the projectivity of $\mathbb P^2$ that takes the line at infinity
to $\ell^+$ maps $\mathcal M$ to a bounded, convex domain of $\mathbb R^2$,
so, we can suppose from now on without loss of generality by \eqref{obs:CentMapsCent},
that $\mathcal M$ is bounded.

First, we reprove \cite[Lemma~1 and Corollary]{KellyStraus} as
\begin{equation}\label{equ:MCisPC}
 \parbox{100mm}{\noindent\itshape%
  For any inner point $O$ in $\mathcal M$,
  there exist two (maybe ideal) points $P$ and $Q$ in $O^*$ such that
  $PQ$ does not intersect $\mathcal M$.
 }
\end{equation}

There is at least one chord $\overline{AC}$ of $\mathcal M$ which is bisected by $O$.
Then, the harmonic conjugate $\hat P$ of $O$ with respect to $A$ and $B$,
is on the line at infinity.

If $O^*$ has a further point at infinity, then let $\hat Q$ be that point.

If $O^*$ has only $\hat P$ at infinity, then $O^*$ is a connected curve,
hence it cannot lie completely within the strip formed by the two supporting lines of
$\mathcal M$ which are parallel to $AC$ because in that case it would intersect $\mathcal M$.
Thus, a point $\hat Q$ of $O^*$ exists outside this strip.

Thus, line $\hat P\hat Q$ does not intersect $\mathcal M$,
but intersects $O^*$ in the points $\hat P$ and $\hat Q$.


Now we prove that
\begin{equation}\label{equ:MCthenPC}
 \parbox{100mm}{\noindent\itshape%
 A point $O\in\mathcal M$ is a center of $(\mathcal M,d)$
 if and only if
 it is a projective center of $\mathcal M$,
  and the metric point reflection $\rho^{}_{O}$
  is $\varpi^{-1}\circ\bar\rho^{}_{\varpi O}\circ\varpi$
  for a proper projectivity $\varpi$.
 }
\end{equation}
If $O$ is a projective center of $\mathcal M$,
then a projectivity  $\varpi$ exists such that
$\varpi(O)$ is an affine center of $\varpi(\mathcal M)$.
Then $\bar\rho^{}_{\varpi(O)}$ is an
involutive isometry with respect to $d'(\cdot,\cdot):=d(\varpi(\cdot),\varpi(\cdot))$,
that keeps the straight line through $\varpi(O)$ invariant.
Thus $\varpi^{-1}\circ\bar\rho^{}_{\varpi O}\circ\varpi$ is
an involutive isometry with respect to $d$,
that keeps the straight line through $\varpi(O)$ invariant.
That is, $O$ is a center of~$(\mathcal M,d)$.

Assume now that $O$ is a center of $(\mathcal M,d)$.
By \eqref{equ:MCisPC} 
we have two (maybe ideal) points $P$ and $Q$ in $O^*$
such that $PQ$ does not intersect $\mathcal M$.

Let $\varpi$ be the projectivity 
that maps line $PQ$ into the ideal line.
Then $\varpi(O)$ is the affine midpoint
of the chords $\overline{AC}:=\varpi(OP)\cap\varpi(\mathcal M)$
and $\overline{BD}:=\varpi(OQ)\cap\varpi(\mathcal M)$.
%
With this is mind,
\eqref{obs:CentMapsCent} allows us to assume
without loss of generality that
$O$ is the affine midpoint of two chords.
Let these chords be $\overline{AC}$ and $\overline{BD}$.

As $\rho^{}_{d;O}$ and $\bar\rho^{}_{O}$ 
are both restrictions of their corresponding unique collineations
\cite[Theorem 3.1.]{CCMECallum2007},
and these collineations coincide on points $A,C,B,D$ and $O$
three of which are in general position,
hence $\rho^{}_{d;O}\equiv\bar\rho^{}_{O}$ follows.
\def\loose{\color{red}!!!!!!!!!!!!!!!!!!!!
As $\rho^{}_{d;O}$ is an isometry of $d$,
it takes chords of $\mathcal M$ into chords of $\mathcal M$,
and it keeps every chord passing through~$O$.
Observe, that these properties are also valid
for the affine point reflection $\bar\rho^{}_{O}$.

As $\rho^{}_{d;O}(\overline{AB})=\overline{CD}=\bar\rho^{}_{O}(\overline{AB})$,
if $X\in\overline{AB}$, then we have
\begin{align*}
  \rho^{}_{d;O}(X)
   &=\rho^{}_{d;O}(\overline{AB}\cap\widetilde{OX})
    =\rho^{}_{d;O}(\overline{AB})\cap\rho^{}_{d;O}(\widetilde{OX})
    =\overline{CD}\cap\widetilde{OX}
  \\
   &=\bar\rho^{}_{O}(\overline{AB})\cap\bar\rho^{}_{O}(\widetilde{OX})
    =\bar\rho^{}_{O}(\overline{AB}\cap\widetilde{OX})=\bar\rho^{}_{O}(X),
\end{align*}
that, with similar deduction for the other sides of the parallelogram $\mathcal P:=ABDCD$,
shows that $\rho^{}_{d;O}=\bar\rho^{}_{O}$ on~$\mathcal P$.

\goodbreak
If $X$ is inside $\mathcal P$,
then there is a point $P$ on $\mathcal P$ such that $X\in\overline{PA}$.
Then we have
\begin{align*}
  \rho^{}_{d;O}(X)
   &=\rho^{}_{d;O}(\widetilde{AP}\cap\widetilde{OX})
    =\rho^{}_{d;O}(\widetilde{AP})\cap\rho^{}_{d;O}(\widetilde{OX})
    =\widetilde{C\rho^{}_{d;O}(P)}\cap\widetilde{OX}
  \\
   &=\widetilde{C\bar\rho^{}_{O}(P)}\cap\widetilde{OX}
    =\bar\rho^{}_{O}(\widetilde{AP})\cap\bar\rho^{}_{O}(\widetilde{OX})
    =\bar\rho^{}_{O}(\widetilde{AP}\cap\widetilde{OX})
    =\bar\rho^{}_{O}(X),
\end{align*}

If $X\in\mathcal M$ is outside $\mathcal P$,
then two of the chords
$\widetilde{AX}$, $\widetilde{BX}$, $\widetilde{CX}$, and $\widetilde{BX}$
intersect one of the edges of $\mathcal P$,
say $\widetilde{AX}$ 
intersects $\overline{CD}$ in a point $P$. 
Letting $R:=\rho^{}_{d;O}(P)=\bar\rho^{}_{O}(P)$
we obtain
\begin{align*}
  \rho^{}_{d;O}(X)
   &=\rho^{}_{d;O}(\widetilde{AP}\cap\widetilde{OX})
    =\rho^{}_{d;O}(\widetilde{AP})\cap\rho^{}_{d;O}(\widetilde{OX})
    =\widetilde{CR}\cap\widetilde{OX}
  \\
   &=\widetilde{C\bar\rho^{}_{O}(P)}\cap\widetilde{OX}
    =\bar\rho^{}_{O}(\widetilde{CP})\cap\bar\rho^{}_{O}(\widetilde{OX})
    =\bar\rho^{}_{O}(\widetilde{CP}\cap OX)=\bar\rho^{}_{O}(X).
\end{align*}
}
This proves \eqref{equ:MCthenPC}.

As every point of $\mathcal M$ is a center of $(\mathcal M,d)$,
from \eqref{equ:MCthenPC} it follows that
every point of $\mathcal M$ is a projective center,
hence \cite[Theorem~3.3(a)]{MontejanoMorales2003} gives that
$\mathcal M$ is an ellipse.

Lemma~\ref{lem:CCTrans}
, 
\eqref{equ:MCthenPC}, 
and
Lemma~\ref{lem:PMonSegment} give 
that $d(X,Y)=c^{}_\ell h(X,Y)$,
where $h$ is the Hilbert metric on $\mathcal M$ and $c^{}_{\ell}$
is a constant.


Consider the different chords $\overline{AB}$ and $\overline{CD}$ of $\mathcal M$,
where $A,B,C,D\in\partial\mathcal M$.


If $\overline{AB}\cap\overline{CD}=\emptyset$,
then one of the intersections
$\overline{AC}\cap\overline{BD}$ or $\overline{AD}\cap\overline{BC}$
is not empty, and that intersection point $O$
is such that $\rho^{}_{d;O}(\overline{AB})=\overline{CD}$,
hence $c^{}_{AB}=c^{}_{CD}$, i.e.\
$c^{}_{\ell}=c^{}_{\ell'}$ if $(\ell\cap\ell')\cap\mathcal M=\emptyset$.
If $\overline{AB}\cap\overline{CD}=\{O\}$,
then let $\overline{EF}$ be a chord of $\mathcal M$ ($E,F\in\partial\mathcal M$)
such that
it does not intersects the quadrangle $ACBD$.
Then $c^{}_{AB}=c^{}_{EF}=c^{}_{CD}$ proves that $c^{}_{\ell}$
does not depend on $\ell$, hence it is a constant~$c$.

The proof of the theorem is complete.
\end{proof}





\section{Finitely many centers}\label{sec:finCenters}

We prove that some finitely many well-placed centers
are enough to deduce the symmetry of the straight projective-metric spaces.


\begin{lemma}\label{lem:irratPRdense}
  If $d(O,P)/d(O,Q)$ is an irrational number
  for the collinear centers $O,P,Q$ of a straight projective-metric space $(\mathcal M,d)$,
  then every point of the common geodesic $\tilde\ell$ of $O,P,Q$
  is a center of $(\mathcal M,d)$.
\end{lemma}
\begin{proof}
By Lemma~\ref{lem:PRSclosed} we need only to prove that the set of centers
on $\tilde\ell$ is dense.

We may assume without loss of generality that $P\in\overline{OQ}$.

As the projective-metric space is straight,
there exists an isometry $\imath$ from $\tilde\ell$ to $\mathbb R$
such that $\imath(O)=0$, and hence $\imath(P)=d(O,P)=:p$ and $\imath(Q)=d(O,Q)=:q$.
By our assumption we have $0<p<q$,
and the condition of the lemma gives that $p/q$ is an irrational number.
Then, Kronecker's Approximation Theorem \cite{WP-KroneckerAT} gives that
for any
$x\in\mathbb R$ and $\varepsilon>0$ there are $i,j\in\mathbb Z$,
such that $|ip-jq-x|<\varepsilon$.

Letting $\tau^{}_{OP}:=\rho^{}_{d;P}\circ\rho^{}_{d;O}$ and
$\tau^{}_{OQ}:=\rho^{}_{d;Q}\circ\rho^{}_{d;O}$ as before,
we obtain by Lemma~\ref{lem:segments}\ref{enu:transLength}, that
$d(\tau^{}_{OP}(X),X)=2p$ and $d(\tau^{}_{OQ}(X),X)=2q$
for any point $X\in\tilde\ell$.
Thus, we obtain
$ 
 \imath\circ\tau^{}_{OP}\circ\imath^{-1}\colon x\mapsto x+2p
 \ \text{ and }\ 
 \imath\circ\tau^{}_{OQ}\circ\imath^{-1}\colon x\mapsto x+2q.
$ 
This means that the set $\mathcal S:=\{\tau^{i}_{OP}(\tau^{j}_{OQ}(O)):i,j\in\mathbb Z\}$
is dense in $\tilde\ell$.
However, Lemma~\ref{lem:genCent} implies
$\tau^{}_{OX}\circ\rho^{}_{d;Y}\circ\tau^{}_{XO}=\rho^{}_{d;\tau^{}_{OX}(Y)}$
for any centers $X,Y\in\tilde\ell$, so every point in $\mathcal S$ is a center.
This proves the Lemma.
\end{proof}

We say that the different points $O$ and $P_i,Q_i$ ($i=1,\dots,k$)
form a \emph{pencil} with \emph{tip} $O$ 
if the points $O,P_i,Q_i$ are collinear for every~$i$.
Such a pencil is called \emph{$l$-dimensional} if the linear space
generated by the affine vectors $\overline{OP_i}$ is $l$-dimensional.

\begin{lemma}\label{lem:pencPMcent}
 In a neighborhood of a center $O$ of a 
 straight projective-metric space
 every point of the affine hyperplane $\mathcal H$
 spanned by the pencil of centers $P_i,Q_i$ ($i=1,\dots,k$)
 and tip $O$ is a center,
 if $d(O,P_i)/d(O,Q_i)$ is irrational for every~$i$.
\end{lemma}
\begin{proof}
We prove by induction.
We consider the $n$-dimensional  straight projective-metric space $(\mathcal M,d)$.

By Lemma~\ref{lem:irratPRdense} we know that all points
of the geodesics $\tilde\ell_i:=OP_i=P_iQ_i$ ($i=1,\dots,n$)
are centers of $(\mathcal M,d)$.

Assume now that for every $l$-dimensional pencil of the given type
the statement of the lemma is fulfilled.

Let the $(l+1)$-dimensional pencil $\mathcal P_{l+1}$ of centers $P_i,Q_i$ and tip $O$
be such that $d(O,P_i)/d(O,Q_i)$ is irrational ($i=1,\dots,k\le n$),
where we clearly have $k\ge l+1$.

If $k>l+1$, then
the pencil of $P_i,Q_i$ and tip $O$ for $i=1,\dots,k-1$
can be either of dimension $l+1$ or of dimension $l$.
In the former case remove the geodesic $\tilde\ell_k:=OP_k=P_kQ_k$,
and continue this procedure until no removing is possible.
This way we can assume that
the pencil $\mathcal P_{l+1}$ 
is such that the pencil $\mathcal P_l$ of $P_i,Q_i$ and tip $O$ ($i=1,\dots,k-1$)
is of dimension~$l$.

By the hypothesis of the induction
there is a 
neighborhood $\mathcal U_l$ of $O$ in the hyperplane $\mathcal H_l$
spanned by the pencil $\mathcal P_l$,
where every point is a center of the projective-metric space.
Further, every point of the geodesic $\tilde\ell_k:=OP_k=P_kQ_k$
is a center by Lemma~\ref{lem:irratPRdense}.

Let $\mathcal O$ be a suitably small neighborhood of~$O$

Let $\mathcal H_l^X$ be the affine subspace spanned
by the points of $\rho^{}_X(\mathcal U_l)$
for every point $X\in\tilde\ell_k\cap\mathcal O$.
Then every point $P\in\rho^{}_X(\mathcal U_l)$ is a center by Lemma~\ref{lem:genCent}.
Let $\mathcal P$ be the common plane of $\ell_k$ and $P$,
and let $Q\in\mathcal P\cap\mathcal U_l$.
Then the geodesic $\widetilde{QP}$ contains at least two centers,
namely $Q$ and $P$.

Let $\mathcal O_l^X$ be an open set in $\rho^{}_X(\mathcal U_l)$ containing $\rho^{}_X(O)$.

Let the points $P\in\mathcal O_l^X$ and  $Q\in\rho^{}_X(\mathcal O_l^X)$ be
such that the geodesic $\widetilde{QP}$ contains a point that is not a center.
If there are no such points,
then the hypothesis of the induction follows for $l+1$, 
that proves the statement of the lemma.

As $\rho^{}_\cdot(\cdot)$ is continuous in its subscript and $\mathcal O_l^X$ is open,
there is a (small) $\varepsilon>0$
such that $\rho^{}_Y(\mathcal U_l)$ intersects $\widetilde{QP}$ in a point $P_Y$
if $Y\in\mathcal Y:=\{Y\in\tilde\ell_k:d(X,Y)<\varepsilon\}$.
Observe that $P_Y$ depends on $Y$ continuously,
hence it either runs over a closed open segment $\mathcal S$ or it is a fixed point $P$.

As there is a point on $\widetilde{QP}$ that is not a center,
the ratio $d(Q,P)/d(Q,P_Y)$ is rational for every $P_Y$
by Lemma~\ref{lem:pencPMcent}, hence $P_Y$ is a fixed point.
Moreover, $P_Y\equiv P$, because $P_X=P$.

Thus,
every point $Z$ of the open triangle $\mathcal Z$ spanned by $\mathcal Y$ and $P$ is a center,
hence every point of the geodesics $\widetilde{QZ}$ is a center.
If $Z\to P$ in $\mathcal Z$,
the geodesic $\widetilde{QZ}$ tends to $\widetilde{QP}$,
and therefore, every point of $\widetilde{QP}$ is a center by Lemma~\ref{lem:PRSclosed}.
This is a contradiction, hence
every point of every geodesic $\widetilde{QP}$ is a center,
if $P\in\mathcal O_l^X$.
This proves the hypothesis of the induction, 
hence the statement of the lemma.
\end{proof}

The following result can be seen as a specific generalization of \cite[(51.5)]{Busemann1955}.

\begin{theorem}\label{thm:strPMcent}
 The set of the centers of an $n$-dimensional straight projective-metric space $(\mathcal M,d)$
 contains an $n$-dimensional pencil of points $P_i,Q_i$ ($i=1,\dots,k\ge n$) and tip $O$
 such that $d(O,P_i)/d(O,Q_i)$ is irrational for every~$i$
 if and only if it is either a Minkowskian or the hyperbolic geometry.
\end{theorem}
\begin{proof}
As every point of 
a Minkowskian or the hyperbolic geometry is a center,
we need only to prove the reverse statement of the theorem.

Assume that
the set of the centers of $(\mathcal M,d)$
contains an $n$-dimensional pencil of points $P_i,Q_i$ ($i=1,\dots,k\ge n$) and tip~$O$.

By Lemma~\ref{lem:pencPMcent},
this assumption implies that the set of the centers of $(\mathcal M,d)$
contains a neighborhood of~$O$,
which by theorems \ref{thm:PPMcent} and \ref{thm:HPMcent}
proves the desired result.
\end{proof}

For projective-metric spaces of parabolic type
or of hyperbolic type containing no affine line
we need less centers to deduce that the metric is
Minkowskian or hyperbolic.

\begin{theorem}\label{thm:parPMcentMinkowski}
 The set of the centers of an $n$-dimensional projective-metric space
 of parabolic type
 contains
 $n+1$ affinely independent point and an additional one
 affinely independent from the others over the rational numbers
 if and only if it is a Minkowski geometry.
\end{theorem}
\begin{proof}
As every point of any Minkowski geometry is a center,
we need only to prove the reverse statement of the theorem.

By \eqref{equ:PMCisAC},
if $O$ is a center, then $\rho^{}_O\equiv\bar\rho^{}_O$.
The product of any two affine point reflections
is an affine translation,
so Kronecker's Approximation Theorem \cite{WP-KroneckerAT}
gives that the centers generated by repeated applications
of the metric point reflections,
form a dense set in $\mathbb R^n$.

Then Lemma~\ref{lem:PRSclosed} and Theorem~\ref{thm:strPMcent}
imply the statement of the Theorem.
\end{proof}

\begin{theorem}\label{thm:hypPMcentBolyai}
 The set of the centers of an $n$-dimensional projective-metric space
 of hyperbolic type with no affine line inside
 contains an $(n-1)$-dimensional pencil of points $P_i,Q_i$ ($i=1,\dots,k\ge n-1$) and tip $O$
 such that $d(O,P_i)/d(O,Q_i)$ is irrational for every~$i$
 if and only if it is the hyperbolic geometry.
\end{theorem}
\begin{proof}
As every point of the hyperbolic geometry is a center,
we need only to prove the reverse statement of the theorem.

Assume that
the set of the centers of
the $n$-dimensional projective-metric space $(\mathcal M,d)$
of hyperbolic type with no affine line inside
contains an $(n-1)$-dimensional pencil of points $P_i,Q_i$ ($i=1,\dots,k\ge n-1$) and tip~$O$.

By Lemma~\ref{lem:pencPMcent},
this assumption implies that the set of the centers of $(\mathcal M,d)$
contains a neighborhood of~$O$ in an $(n-1)$-dimensional hyperplane $\mathcal H$.
By Lemma~\ref{lem:genCent} this means that every point of $\mathcal M\cap\mathcal H$
is a center, which, by \eqref{equ:MCthenPC}, means that
every point of $\mathcal M\cap\mathcal H$ is a projective center of~$\mathcal M$.
According to \cite[Theorem~3.3(a)]{MontejanoMorales2003},
this implies that $\mathcal M$ is an ellipsoid, hence the theorem.
\end{proof}


\def\loose
{
\begin{corollary}\label{thm:MCPCNC}
 The following are equivalent conditions for point $O\in\mathcal I$
 of the Hilbert geometry $(\mathcal I,d_{\mathcal I})$:
 \begin{enumerate}
  \item $O$ is of non-positive curvature;
  \item $O$ is a projective center of $\mathcal I$;
  \item $O$ is a metric center of $(\mathcal I,d_{\mathcal I})$.
 \end{enumerate}
\end{corollary}
}

\begin{acknowledgement}
 The author thanks J\'anos Kincses and Tibor \'Odor
 for the oral discussions on the subject of this paper.
 Thanks are also due to the anonymous referee
 for raising the problems solved in Section~\ref{sec:finCenters}. 
\end{acknowledgement}


\begin{thebibliography}{99}




\bibitem{BusemannKelly1953}\scshape H.~Busemann \text{\rm and} P.~J.~Kelly,
 \itshape Projective Geometries and Projective Metrics,
  \upshape Academic Press, New York, 1953. 

\bibitem{Busemann1955}\scshape H.~Busemann,
 \itshape The Geometry of Geodesics,
  \upshape Academic Press, New York, 1955.


\bibitem{Busemann1970}\scshape H. Busemann,
 \itshape {Recent Synthetic Differential Geometry\rm,
  Ergebnisse der Mathematik und ihrer Grenz\-ge\-biete, \bfseries54},
   \upshape Springer, New York, 1970.

\bibitem{CCMECallum2007}\scshape A. \u{C}ap, M.~G. Cowling, F. de Mari, M. Eastwood, {\rm and} R.
McCallum,
 \upshape The Heisenberg group, SL(3, R), and rigidity,
  \itshape Harmonic Analysis, Group Representations, Automorphic Forms and Invariant Theory,
   \upshape Lecture Notes Series, Institute of Mathematical Sciences, National University of Singapore {\bfseries12},
    World Scientific, New Jersey -- London, 2007.














\bibitem{Hirschfeld1979}\scshape J.~W.~P. Hirschfeld,
 \itshape Projective Geometries Over Finite Fields,
  \upshape Clarendon Press,  Oxford, 1979. 




\bibitem{KellyStraus}\scshape P.~J. Kelly \text{\rm and} E. Straus,
 \upshape Curvature in Hilbert Geometries,
  \itshape Pacific J. Math.\upshape, {\bfseries8} (1958), 119--125.







\bibitem{MontejanoMorales2003}\scshape L. Montejano \text{\rm and} E. Morales,
 \upshape  Characterization of ellipsoids and polarity in convex sets,
  \itshape Mathematika\upshape, {\bfseries50} (2003), 63--72;
   DOI: \href{https://doi.org/10.1112/S0025579300014790}{10.1112/S0025579300014790}.


\bibitem{Szabo}\scshape Z. I. Szab\'o,
 \upshape Hilbert's fourth problem. I,
  \itshape Adv. in Math.\upshape, {\bfseries59:3} (1986), 185--301;
   DOI: \href{http://dx.doi.org/10.1016/0001-8708(86)90056-3}{10.1016/0001-8708(86)90056-3}.









\bibitem{WP-CauchyEqu} \url{https://en.wikipedia.org/wiki/Cauchy's_functional_equation}

\bibitem{WP-KroneckerAT} \url{https://en.wikipedia.org/wiki/Kronecker's_theorem}


\end{thebibliography}
\end{document}